\documentclass[a4paper,12pt]{amsart}
\usepackage[mathscr]{eucal}
\usepackage{amssymb}
\usepackage{latexsym}
\usepackage{enumerate}

\theoremstyle{plain}
\newtheorem{thm}{Theorem}[section]
\newtheorem{prop}[thm]{Proposition}
\newtheorem{lem}[thm]{Lemma}
\newtheorem{cor}[thm]{Corollary}
\newtheorem{conj}[thm]{Conjecture}
\theoremstyle{definition}
\newtheorem{defn}[thm]{Definition}
\newtheorem{rem}[thm]{Remark}

\evensidemargin\paperwidth
\advance\evensidemargin-\textwidth
\advance\oddsidemargin-1in
\evensidemargin\oddsidemargin

\topmargin\paperheight
\advance\topmargin-\textheight
\advance\topmargin-1in

\begin{document}

\title[Fixed point property and bounded cohomology]{Fixed point properties and second bounded cohomology of  universal lattices on Banach spaces}
\author{Masato MIMURA}
\date{}
\thanks{The author is 
supported by JSPS Research Fellowships for Young Scientists No.20-8313. }

\maketitle

\begin{abstract}
Let $B$ be any $L^p$ space for $p\in (1,\infty)$ or any Banach space isomorphic to a Hilbert space, and $k \geq 0$ be integer. We show that if $n\geq 4$, then the universal lattice $\Gamma =SL_n (\mathbb{Z}[x_1, \ldots , x_k])$ has property $(\mathrm{F}_B)$ in the sense of Bader--Furman--Gelander--Monod. Namely, any affine isometric action of $\Gamma$ on $B$ has a global fixed point. The property of having $(\mathrm{F}_B)$ for all $B$ above is known to be strictly stronger than Kazhdan's property $(\mathrm{T})$. We also define the following generalization of property $(\mathrm{F}_B)$ for a group: the boundedness property of all affine quasi-actions on $B$. We name it property $(\mathrm{FF}_B)$ and prove that the group $\Gamma$ above also has this property modulo trivial part. The conclusion above implies that the comparison map in degree two $H^2 _b (\Gamma; B) \to H^2(\Gamma; B)$ from bounded to ordinary cohomology is injective, provided that the associated linear representation does not contain the trivial representation.

\end{abstract}

\section{\textbf{Introduction and main results}}\label{sec:Intro}
Kazhdan's property (T), which was first introduced in \cite{Kazh}, represents 
certain forms of \textit{rigidity} of a group, and now plays an important role in wide range of mathematical fields: for instance, see a book of Bekka--de la Harpe--Valette \cite{BHV}. Property (T) is 
initially defined in terms of unitary representations. 
Recall that a group $\Gamma$ is defined to have \textit{property} $(T)$ 
if any unitary representation 
$(\pi , \mathcal{H})$ of $\Gamma$ does not admit almost invariant vectors 
in $\mathcal{H}_{\pi (\Gamma)} ^{\perp}$, which means the orthogonal complement of the subspace of all $\pi (\Gamma)$-invariant vectors. Here a representation $\rho$ of a group $\Gamma$ on a Banach space $B$ is said to admit 
\textit{almost invariant vectors} if for any compact subset  
$F\subset \Gamma$ and any $\varepsilon >0$, there exists a vector $\xi$ in the unit sphere $S(B)$ of $B$ 
such that $\sup_{s \in F} \| \rho (s)\xi -\xi \| \leq \varepsilon$ holds.   P. Delorme \cite{De} and A. Guichardet \cite{Gu} 
have shown that for any locally compact and second countable group $\Gamma $, 
property (T) is equivalent to Serre's \textit{property} $(FH)$: 
a group $\Gamma$ is said to have (FH) if any affine isometric action of $\Gamma$ on a 
Hilbert space has a global fixed point.

In 2007, Bader--Furman--Gelander--Monod \cite{BFGM} investigated similar 
properties in the broader framework of general Banach spaces $B$. 
They named 
the Kazhdan type property and the fixed point property respectively 
$\mathit{(T_B)}$ and $\mathit{(F_B)}$. (We will recall the precise definitions in Section~\ref{sec:Pre}.) 
For convenience, we shall use the following symbols for certain classes of Banach spaces. 
\begin{itemize}
   \item  The symbol $L^p$ denotes the class of all $L^p$ spaces (on arbitrary measures). 
   \item  The symbol $[\mathcal{H}]$ denotes the class of all Banach spaces which are isomorphic to Hilbert spaces (, namely, which have compatible norms to ones of Hilbert spaces). 
\end{itemize}
In \cite{BFGM}, Bader--Furman--Gelander--Monod proved the following theorem and revealed that $(\mathrm{F}_B)$ is 
stronger than $(\mathrm{T}_B)$ in general. 
We note that the statement for $[\mathcal{H}]$ in $(iii)$ in the theorem is due to Y. Shalom  in his unpublished work. (We mention that the original statement requires $\sigma$-finiteness of measure spaces in $(ii)$, but the argument in \cite[\S 4.b]{BFGM} works for general cases.)
\begin{thm}\label{thm:BFGM}$($\cite[Theorem A and Theorem B]{BFGM}$)$
Let $G$ be a locally compact and second countable group. 
\begin{enumerate}[$\mathrm{(}$i$\mathrm{)}$]
   \item  For any  Banach space $B$, 
        property $(\mathrm{F}_B)$ implies property $(\mathrm{T}_B)$.
   \item  Property $(\mathrm{T})$ is equivalent to property 
   $(\mathrm{T_{\mathit{L^p}}})$, where $p \in (1, \infty)$. 
   It is also equivalent to 
    property $(\mathrm{F_{\mathit{L^p}}})$, 
    where $p \in (1, 2]$.   
    \item  Suppose that $G= \Pi_{i=1}^{m} \mathbf{G}_i(k_i)$, where $k_i$ are 
       local fields and $\mathbf{G}_i(k_i)$ are $k_i$-points of Zariski 
       connected simple $k_i$-algebraic groups. If each simple factor 
       $\mathbf{G}_i(k_i)$ has $k_i$-rank $\geq 2$, 
       then $G$ and the lattices in $G$ have property 
       $(\mathrm{F_{\mathit{L^p}}})$ for $1<p<\infty$ and 
       property $(\mathrm{F}_{[\mathcal{H}]})$.
\end{enumerate}
\end{thm}

They also noted that both of the properties 
``$(\mathrm{F}_{L^p})$ for all $2<p<\infty$" 
and $(\mathrm{F}_{[\mathcal{H}]})$ are strictly stronger than $(\mathrm{T})$. 
In fact, 
G. Yu \cite{Yu} has proved that any hyperbolic group, including 
one with (T), admits a proper affine isometric action on 
some ${\ell}^p$ space. Existence of a proper affine action 
represents \textit{opposite} nature to rigidity for a group. 
Hence higher rank algebraic groups and lattices have 
stronger rigidity than hyperbolic (T) groups do. We also note that Shalom has announced 
that $Sp (n,1)$ fails to have $(\mathrm{T}_{[\mathcal{H}]})$. 
The author does not know whether there exists an infinite hyperbolic group 
with $(\mathrm{F}_{[\mathcal{H}]})$ or $(\mathrm{T}_{[\mathcal{H}]})$. (Note that $(\mathrm{F}_{[\mathcal{H}]})$ is equivalent to the property that any affine \textit{uniformly bi-Lipschitz} action on a Hilbert space has a global fixed point, and in \cite{BFGM} this property is written as \textit{property} $(\overline{F}_{\mathcal{H}})$. Similarly, in \cite{BFGM} $(\mathrm{T}_{[\mathcal{H}]})$ is written as \textit{property} $(\overline{T}_{\mathcal{H}})$.)

From the backgrounds above, 
 it seems to be a significant problem to establish 
 $(\mathrm{F_{\mathit{L^p}}})$ $(1<p<\infty )$ and 
$(\mathrm{F}_{[\mathcal{H}]})$ for certain groups. 
However as far as the author knows, the only known examples 
were the following: higher rank algebraic groups and lattices therein as in $(iii)$ of Theorem \ref{thm:BFGM}; and certain limit group of random groups \cite{NS} in the sense of M. Gromov. 
One of the main results of this paper is 
to provide a new example of groups with the properties above. Our example is of the form $SL_n(\mathbb{Z} [x_1, \ldots , x_k])$ ($n\geq 4$ and $k$ finite), and it is called the \textit{universal lattice} by Shalom \cite{Shal1}. We note that for $n\geq 3$, property (T) for universal lattices has been proved by Shalom \cite{Shal3} and L. Vaserstein\cite{Vas}. However, at the present we have to exclude the case of $n=3$ for establishing $(\mathrm{F_{\mathit{L^p}}})$ $(p>2)$ and $(\mathrm{F}_{[\mathcal{H}]})$.

\begin{thm}\label{thm:FSLnZ}
Let $k \geq 0$ be an integer. Then for $n \geq 4$, the universal lattice 
$SL_n(\mathbb{Z} [x_1, \ldots , x_k])$ has property $(\mathrm{F}_{\mathcal{C}})$. Here $\mathcal{C}$ stands for either the class $L^p$ $(1<p<\infty )$ or the class $[\mathcal{H}]$. 
\end{thm}

We note that this theorem particularly implies property $(\mathrm{T}_{[\mathcal{H}]})$ 
of universal lattices with $ n\geq 4$. It follows from $(i)$ of Theorem 
\ref{thm:BFGM}. We also mention that Shalom call these groups universal lattices because for any unital, commutative finitely generated ring $A$, some universal lattice surjects onto the \textit{elementary linear group} $EL_n (A)$. Here we define $EL_n (A)$ as the multiplicative group of $n\times n$ matrices generated by elementary matrices. The Suslin stability theorem \cite{Sus} states that for $A_k =\mathbb{Z}[x_1,...,x_k]$, $EL_n (A_k)$ coincides with $SL_n (A_k)$ if $n\geq 3$. We note that universal lattices can\textit{not} be realized as an arithmetic lattice of any algebraic group because of the Margulis normal subgroup theorem. Interest in universal lattice from aspect of rigidity dates back to one by A. Lubotzky.

For our proof of Theorem \ref{thm:FSLnZ}, 
we need to deduce $(\mathrm{F}_B)$ from $(\mathrm{T}_B)$ even though this direction does not hold in general. 
There are the following two well-known cases in which the direction above 
is true: first, 
 the case of that $B=\mathcal{H}$ is proved by Delorme \cite{De} with 
the aid of theory of conditionally negative definite functions. Second, the 
case of a higher rank 
algebraic group is treated in \cite[\S 5]{BFGM}. In this case, 
the \textit{Howe--Moore property} for simple algebraic groups \cite[Appendix 9]{BFGM} is the key. (Bader--Furman--Gelander--Monod have employed this property to show higher rank algebraic groups have $(\mathrm{F}_{L^p})$, and then deduce the case of higher rank lattices by induction. Therefore their proof for higher rank lattices is not direct.) 
By making use of the \textit{relative} versions of $(\mathrm{T}_B)$ and 
$(\mathrm{F}_B)$, we have shown the 
following \textit{new} implication in this direction: 

\begin{thm}\label{thm:TtoF}
Let $k \geq 0$ be an integer and $A_k = \mathbb{Z} [x_1, \ldots , x_k]$. 
Suppose $B$ is a superreflexive Banach space.  If the pair 
$ EL_2 (A_k) \ltimes A_{k}^2 \triangleright A_{k}^2$ has relative property 
$(\mathrm{T}_B)$, then the pair $ SL_3 (A_k ) \ltimes A_{k}^3 \triangleright A_{k}^3$ 
has relative property $(\mathrm{F}_B)$.
\end{thm}

To prove Theorem \ref{thm:FSLnZ}, we combine Theorem \ref{thm:TtoF} with 
the following relative $(\mathrm{T}_B)$, Shalom's argument in \cite{Shal3}, and Vaserstein's 
bounded generation in \cite{Vas}.

\begin{thm}\label{thm:Tsemi}
With the same notation as one in  Theorem \ref{thm:TtoF}, the pair 
$ EL_2 (A_k) \ltimes A_{k}^2 \triangleright A_{k}^2$ has relative 
property $(\mathrm{T}_{\mathcal{C}})$. Here $\mathcal{C}$ stands for 
$L^p$ $(1<p<\infty)$ or $[\mathcal{H}]$.
\end{thm}

Further, we generalize property $(\mathrm{F}_B)$ by taking a \textit{quasification}, in a way similar to one by N. Monod \cite{Mon}. More precisely, we define that a group has \textit{property }$(FF_B)$ if any \textit{quasi-action} on $B$ of the group has bounded orbits. Here in this paper, the quasification means we become to allow uniformly bounded error from the original condition. 
Our next result is that the universal lattices for $n\geq 4$ have the following limited version of this property $(\mathrm{FF}_B)$ for $B \in \bigcup_{1<p<\infty } L^p \cup [\mathcal{H}]$: for any quasi-action on $B$ whose associated linear representation does \textit{not} have non-zero invariant vectors, any orbit is bounded.  We define the \textit{property} $(FF_B)/T$ (``\textit{property }$(FF_B)$ \textit{modulo trivial part}") as a certain slightly stronger form of the limited property above. The exact definition and precise arguments will be taken place in Section \ref{sec:FFB}. (We mention that the following result might be \textit{new} even for the case of that $B=\mathcal{H}$.)

\begin{thm}\label{thm:FFB}$($Main Theorem$)$
Let $k \geq 0$ be an integer and $n \geq 4$. Then 
universal lattice 
$SL_n (\mathbb{Z}[x_1 , \ldots , x_k])$ has property 
$(\mathrm{FF}_{L^p})/\mathrm{T}$ $(1<p< \infty)$ and property 
$(\mathrm{FF}_{[\mathcal{H}]})/\mathrm{T}$.
\end{thm}

For any superreflexive Banach space $B$, ``property $(\mathrm{FF}_B)/\mathrm{T}$" implies property $(\mathrm{F}_B)$. The author does not know whether our  argument can be extended to the case of that $B$ is not superreflexive. We note that V. Lafforgue \cite{Laf1}, \cite{Laf2}, Corollaire 0.7, Proposition 5.2 and Proposition 5.6, has shown that  $\mathrm{SL}_3 (F)$ $($, where $F$ is a non-Archimedean local field$)$ and cocompact lattices therein have $(\mathrm{F}_B)$ 
for any Banach space $B$ of type $>1$. He has asked whether these group have $(\mathrm{FF}_B)$ for a Banach space $B$ of type $>1$(or more generally, of finite cotype). We note that the boundedness property $(\mathrm{FF}_B)$ does \textit{not} necessarily imply the fixed point property $(\mathrm{F}_B)$ if $B$ is not superreflexive.

We apply Theorem \ref{thm:FSLnZ} and Theorem \ref{thm:FFB} to the following two objects: actions on the circle; and second bounded cohomology. We note that for the same class $\mathcal{C}$ as in Theorem \ref{thm:FSLnZ}$, (\mathrm{F}_\mathcal{C})$ and $(\mathrm{FF}_\mathcal{C})$ (and $(\mathrm{FF}_\mathcal{C})/\mathrm{T}$) pass to quotient groups and subgroups of finite indices. (For the heredity to subgroups of finite indices, one uses $p$-induction. See \cite[\S 8]{BFGM}). Therefore the corollaries below hold for $\Gamma=EL_n (A)$ or its subgroup of finite index, where $A$ stands for \textit{any} unital, commutative and finitely generated ring. (For instance, $\Gamma=SL_n(\mathbb{Z}[x,y, x^{-1},y^{-1}])$ and $\Gamma=SL_n(\mathbb{F}_q[x_1,\ldots ,x_k])$, where $q$ is a prime power and $\mathbb{F}_q$ denotes the finite field of order $q$.)

\begin{cor}\label{cor:circle}
Let $k \geq 0$ be an integer and $\mathcal{C}$ be the same class as one in 
Theorem \ref{thm:FSLnZ}. Let $n \geq 4$, and $\Gamma$ be any group of one of the following three forms: 
\begin{enumerate}[$\mathrm{(}$a$\mathrm{)}$]
  \item universal lattices $SL_n (\mathbb{Z} [x_1 , \ldots , x_k])$, 
  \item quotient groups of $($a$)$,
  \item subgroups of $($a$)$ or $($b$)$ of finite indices.
\end{enumerate}
Then for any $\alpha >0$, every homomorphism $\Phi \colon$$\Gamma \to \mathrm{Diff}^{1+\alpha}_{+} (\mathrm{S}^1)$ has finite image. 
\end{cor}
\begin{cor}\label{cor:bddcohom}
Let $\mathcal{C}$ and $\Gamma$ be the same as in Corollary \ref{cor:circle}. Then the comparison map in degree $2$ 
$$
\Psi^2 \colon H_b^2(\Gamma; B,\rho ) \to 
H^2(\Gamma; B,\rho )
$$
 is injective, for any $B\in \mathcal{C}$ and any isometric representation $\rho$ on $B$ which does $\mathrm{not}$ contain the trivial representation. 

\end{cor}

These applications above shall be discussed in Section \ref{sec:appli}. Corollary \ref{cor:circle} states that the group $\Gamma$ cannot act on the circle in non-trivial way with certain regularity condition, and it can be seen as an extension of Navas' theorem \cite{Nav1} for Kazhdan groups for the case of $\alpha > 1/2$. In the case of subgroups of finite indices in $SL_n (\mathbb{Z})$, stronger result, which shows a similar rigidity for orientation preserving homeomorphisms on the circle, is proved by D. Witte \cite[Corollary 2.4]{Wit}. However, the proof needs the Margulis normal subgroup theorem, and one cannot apply the argument to the case of universal lattices. Also, we note that if $\Gamma$ has a torsion, then the proof of Corollary~\ref{cor:circle} is easier. But that by Selberg's lemma, one can take $\Gamma$ as in option (c) without torsions. 

Corollary \ref{cor:bddcohom} can be seen as 
some generalization of \cite[Theorem 21]{BM}. However, in the case of lattices $\Gamma$ of higher rank algebraic groups $G$, much stronger result is known. Indeed, Monod and Shalom \cite[Theorem 1.4]{MS} have proved the theorem that unless $\pi_1(G)$ is infinite and the local field is $\mathbb{R}$, the second bounded cohomology of $\Gamma$ with separable coefficient Banach modules always vanishes. By comparing Corollary \ref{cor:bddcohom} with the theorem above, the author would like to conjecture the following: 
\begin{conj}
Let $\Gamma$ be the same group as in Corollary $\ref{cor:circle}$. 
Then the second bounded cohomology 
of $\Gamma$ vanishes for every separable coefficient Banach module.
\end{conj}

In addition, we have defined the relative Kazhdan constant for property 
$(\mathrm{T}_{[\mathcal{H}]})$ and performed a certain estimate. For details, see Appendix. 

\begin{prop}\label{prop:relconst}
With the same notation as one in  Theorem \ref{thm:TtoF}, let 
$G =EL_2 (A_k) \ltimes A_{k}^2 $ and $N=A_{k}^2$. Set $F$ be 
the set of all unit elementary matrices in 
$G$ $(\subset SL_3(A_k))$. 
Then the inequality 
\begin{equation*}
\overline{\mathcal{K}}(G , N ; F;  M) > 
(15k +100)^{-1}M^{-6} 
\end{equation*}
holds.  
In the case of that $k=0$, 
$
\overline{\mathcal{K}}(SL_2(\mathbb{Z})\ltimes \mathbb{Z}^2 , \mathbb{Z}^2 ; F;  M) > (21M^6)^{-1}
$ holds. Here the symbol $\overline{\mathcal{K}}(G , N ; F;  M)$ denotes the $\mathrm{generalized\ relative}$ $\mathrm{Kazhdan\ constant\ for}$ $\mathrm{uniformly\ bounded\ representations}$, which is defined in Definition \ref{def:extended const}. 
\end{prop}

\textit{Organization of this paper}: In Section~\ref{sec:Pre} we recall defitions of $(\mathrm{T}_B)$ and $(\mathrm{F}_B)$, and other basic tools in this paper. Section~\ref{sec:TtoF} is devoted to the proof of Theorem~\ref{thm:TtoF}. In Section~\ref{sec:relT}, we present an outlined proof of Theorem~\ref{thm:Tsemi}. In Section~\ref{sec:FSLnZ}, we explain the idea, which is originally due to Shalom, of utilizing reduced cohomology, and deduce Theorem~\ref{thm:FSLnZ} from this idea, Vaserstein's bounded generation, Theorem~\ref{thm:Tsemi}, and Theorem~\ref{thm:TtoF}. Section~\ref{sec:FFB} is for quasifications of $(\mathrm{F}_B)$. There we define $(\mathrm{FF}_B)$ and $(\mathrm{FF}_B)/\mathrm{T}$, and prove Theorem~\ref{thm:FFB}. In Section~\ref{sec:appli}, we discuss certain applications (Corollary~\ref{cor:circle} and Corollary~\ref{cor:bddcohom}) of our results. In Appendix, we define the generalized Kazhdan constant for uniformly bounded representations, and provide a quantitative and detailed proof of Theorem~\ref{thm:Tsemi}, namely, that of Proposition~\ref{prop:relconst}.

\bigskip

\noindent
\textit{\textbf{Notation and convention.}}

Throughout this paper, we assume all rings are associative, all representations 
 and actions of a topological group are strongly continuous, and all subgroups of a topological 
group are closed. We also assume \textit{all topological groups in this paper are locally compact and second countable}. We let 
$\Gamma$, $G$ and $N$ be topological groups, $B$ be a Banach 
space, 
$\mathcal{C}$ be a class of Banach spaces, and 
$\mathcal{H}$ be an arbitrary Hilbert space. For a Banach space $B$, we define  $S(B)$ as the unit sphere, $\mathbb{B}(B)$ as the Banach algebra of all bounded linear operators on $B$, and $\langle  \cdot , \cdot \rangle$ as the duality  
$B \times B^* \to \mathbb{C}$. In this paper, we shall define the following properties in terms 
of $B$: \textit{relative }$(T_B)$, $(T_B)$; 
\textit{relative }$(F_B)$, $(F_B)$; the \textit{Shalom property for }$(F_B)$; 
\textit{relative }$(FF_B)$, $(FF_B)$, $(FF_B)/T$; 
and the \textit{Shalom property for }$(FF_B)$. If we let $(P_B)$ represent any of 
these properties, then we define the property $(P_{\mathcal{C}})$ 
in terms of $\mathcal{C}$ as follows: having $(P_{\mathcal{C}})$ stands for 
having $(P_{B})$ for all $B \in \mathcal{C}$.

\section*{acknowledgments}
The author would like to thank his supervisor Narutaka Ozawa for introducing 
him to this topic, and Yasuyuki Kawahigashi for comments. He is grateful to Bachir Bekka for the symbol $[\mathcal{H}]$, to Yves de Cornulier for fruitful comments, to Nicolas Monod for the symbol $(\mathrm{FF}_B)/\mathrm{T}$, to Andr\'{e}s Navas for suggesting him stating Corollary \ref{cor:circle} explicitly, and to Mamoru Tanaka for pointing out a mistake on stability under ultralimits in the previous version of this paper. He also thanks Alex Furman for the reference \cite{Wit}, Uzy Hadad for corrections of some errors, Martin Kassabov for conversations on noncommutative universal lattices, and Vincent Lafforgue for arguments and drawing my attention to non-superreflexive cases. Finally, he would like to express his gratitude to the referee, whose comments have improved this paper.

\section{\textbf{Preliminaries}}\label{sec:Pre}
\subsection{\textbf{Superreflexivity and property}$(\mathrm{T}_B)$}
\begin{defn}$($\cite{BFGM}$)$ \label{def:propertyTB}
Let $B$ be a Banach space. 
\begin{itemize}
  \item A pair  $G \triangleright N$ of groups is 
  said to have \textit{relative property} $\mathit{(T_B)}$  if for any 
  isometric representation $\rho$ of $G$ on $B$, the 
  isometric representation $\rho '$ on the quotient Banach space  $B / B^{\rho (N)}$, naturally induced by $\rho$, does not admit 
  almost invariant vectors. Here $B^{\rho (N)}$ 
  stands for the subspace of $B$ of all $\rho(N)$-invariant vectors. 
  \item  A group $\Gamma$ is said to have \textit{property} 
  $\mathit{(T_B)}$ if $\Gamma \triangleright \Gamma$ has 
  relative $(\mathrm{T}_B)$.
\end{itemize}
\end{defn}

In the case of that $B$ is superreflexive, there exists a natural complement 
in $B$ of $B^{\rho (N)}$ (Proposition~\ref{prop:decomp}). To see this, we start with the definition of superreflexivity. 

\begin{defn}\label{def:modulus}
Let $B$ be a Banach space. 
\begin{itemize}
   \item The space $B$ is said to be \textit{uniformly convex} (or \textit{uc}) if for all 
   $0< \varepsilon <2$,  $d_{\| \cdot \|}(\varepsilon) >0$ holds. Here for  
   $0< \varepsilon <2$, we define 
   \begin{equation*}
   d_{\| \cdot \|} 
   (\varepsilon)= \inf \left\{ 1- \frac{\| \xi +\eta \| }{2} : \| \xi \| \leq 1 , 
   \ 
   \| \eta \| \leq 1 ,\ \mathrm{and}\  \| \xi -\eta \| \geq \varepsilon  \right\}.
   \end{equation*}
   \item The space $B$ is said to be \textit{uniformly smooth} 
   (or \textit{us})
    if $\displaystyle{\lim_{\tau \to 0} r_{\| \cdot \|} (\tau)/ \tau } =0$ 
    holds. 
   Here for $\tau >0$, we define 
   \begin{equation*}
   r_{\| \cdot \|} 
   (\tau)= \sup \left\{ \frac{\| \xi +\eta \| +\| \xi -\eta \| }{2} -1 : 
   \| \xi \| \leq 1 ,\, \| \eta \| \leq \tau \right\}.
   \end{equation*}
   \item  The space $B$ is said to be \textit{ucus} if $B$ is uc and us.
   \item  The space $B$ is said to be \textit{superreflexive} if 
   it is isomorphic to some ucus Banach space.
\end{itemize}
   We call $d$ and $r$ the \textit{modulus of convexity} and that 
   \textit{of smoothness} respectively.
\end{defn}
We refer to a book of Y. Benyamini and J. Lindenstrauss \cite[\S A]{BL} for details on ucus Banach spaces. 

\begin{lem}\label{lem:lindenst}$($\cite{BL}$)$
Let $(B, \| \cdot \|)$ be a Banach space. Then for any $\tau >0$, 
$r_{\| \cdot \|} (\tau)$
$
= \sup_{0< \varepsilon <2} \left\{ \varepsilon \tau /2 -
 d_{\| \cdot \|_{*}}(\varepsilon) \right\}
$. In particular, $B$ is us if and only if $B^*$ is uc.
\end{lem}

\begin{lem}\label{lem:dual}$($\cite{BL}$)$
Let $B$ be a us Banach space. Then for any $\xi\in S(B)$, 
there exists a unique 
element $\xi^* \in S(B^*)$ such that $\langle \xi, \xi^* \rangle  = 1$. 
Moreover, the map $S(B) \to S(B^*);$ $\xi \mapsto \xi^*$ is uniformly continuous. 
We call this map $\xi \mapsto \xi^*$ the $\mathrm{duality\ mapping}$. 
\end{lem}

Any Hilbert space $\mathcal{H}$ is ucus because $
d_{\| \cdot \|_{\mathcal{H}}} (\varepsilon ) = 1 - 
\sqrt{1- ( \varepsilon  / 2 )^2}
$
 and
$
r_{\| \cdot \|_{\mathcal{H}}} (\tau ) = \sqrt{1 + \tau^2  } -1
$. Any $L^p$ space is ucus if $1<p<\infty$, whereas $L^1$ spaces 
and $L^{\infty}$ spaces are not. (Here we assume dimensions $\geq 2$.) 

\begin{rem}\label{rem:superref}
A representation $\rho$ of $\Gamma$ on $B$ is said to be  
\textit{uniformly bounded} if $ | \rho | \colon= \sup_{g \in \Gamma} \| \rho (g ) \|_{\mathbb{B}(B)} <+ \infty$. 
Bader--Furman--Gelander--Monod \cite[Proposition 2.3]{BFGM} have proved that 
any uniformly bounded representation $\rho$ on a superreflexive Banach space 
$B$ is isometric with respect to some ucus compatible norm. 
\end{rem}

For an isometric representation $\rho$ of $\Gamma$  on $B$, we define the \textit{contragredient representation} $\rho^{\dagger}$ of $\Gamma$ on $B^*$ as follows: 
for any $g \in \Gamma$, $\phi \in B^*$ and $\xi \in B$,  $\langle \xi, \rho^{\dagger}(g)\phi \rangle = \langle \rho (g^{-1})\xi, \phi \rangle $. 
If $B$ is us, then the equality 
$\left(\rho (g)\xi \right)^* =\rho^{\dagger} (g) \xi^*$ holds by definition. 

\begin{prop}$\mathrm{(}$\cite[Proposition 2.6]{BFGM}$\mathrm{)}$\label{prop:decomp}
Suppose $G \triangleright N$ and $\rho$ is an isometric representation of 
$G$ on a us space $B$. Let $B_0$ be $B^{\rho (N)}$ and let $B_1=B_{\rho (N)}'$ denote the annihilator 
of $(B^{*})^{\rho^{\dagger}(N)}$ in $B$. Then $B= B_0 \oplus B_1$ is a 
decomposition of $B$ into two $\rho (G)$-invariant subspaces. Furthermore, for 
any $\xi=\xi_0 + \xi_1$ $(\xi_0 \in B_0 ,\xi_1 \in B_1)$, the inequality $\| \xi_0 \| \leq \|\xi \|$ holds.
\end{prop}

\begin{prop}$($\cite[Proposition 2.10]{BFGM}$)$\label{prop:super}
Let $G \triangleright N$ be a group pair. Suppose a Banach space $B$ is superreflexive. Then for any isometric representation $\rho$ of $G$, $B_1=B'_{\rho (N)}$ is isomorphic to $B / B^{\rho (N)}$ as $G$-representations. Particularly, $G \triangleright N$ has relative $(\mathrm{T}_B)$ if and only if no isometric representation $\rho$ admits almost invariant vectors in $B_{\rho (N)}'$. 
\end{prop}

\subsection{\textbf{Property }$(\mathrm{F}_B)$}

An \textit{affine isometric action} $\alpha$ of $\Gamma$ on $B$ is an action of the form 
$\alpha (g ) \xi = \rho(g) \xi + c(g)$. Here $\rho$ is an 
isometric representation  
and $c(g) \in B$. 
We sometimes simply write $\alpha = \rho +c$ . We call $\rho$ and $c$ 
respectively the \textit{linear part} and the 
\textit{transition part} of $\alpha$. 
Because $\alpha$ is an action, the transition part $c$ 
 satisfies the following condition, called the \textit{cocycle identity}:
\begin{equation*}
\mathrm{For\ any\ }g ,h \in \Gamma , \ \ c(gh)=
 c(g) + \rho (g) c(h) .
\end{equation*}

We also call $c$ the \textit{cocycle part} of $\alpha$. 
\begin{defn}\label{def:1-cohom}
For an isometric representation $\rho$ on $B$, we call a map 
$c \colon \Gamma \to B$ a $\rho$\textit{-cocycle} if it  satisfies the 
cocycle identity. We call $c$ a $\rho$\textit{-coboundary} 
if there exists $\xi \in B$ such that $c(g)=\xi- \rho (g)\xi$ for all $g \in \Gamma$. 
We let $Z^1(\Gamma ; B, \rho)$ and $B^1(\Gamma ;B, \rho)$ denote respectively 
the spaces of all $\rho$-cocycles and of all $\rho$-coboundaries. We define the 
\textit{first cohomology} of $\Gamma$ with $\rho$-coefficient as the additive group 
$
H^1 (\Gamma ;B, \rho) = Z^1(\Gamma ;B, \rho) / B^1(\Gamma ;B, \rho)
$. 
\end{defn}
The space $Z^1(\Gamma ;B, \rho)$ is a Fr\'{e}chet space with respect to its natural topology. Namely, the uniform convergence topology on compact subsets of $\Gamma$. However the coboundary $B^1(\Gamma ;B, \rho)$ is \textit{not} closed in general. We shall examine details in Section \ref{sec:FSLnZ}.

\begin{defn}$($\cite{BFGM}, for the second case$)$ \label{def:property(FH)}
Let $B$ be a Banach space. 
\begin{itemize}
  \item A pair  $G>N$ of groups is 
  said to have \textit{relative property} $\mathit{(F_B)}$ if any 
  affine isometric action of $G$ on $B$ has an 
  $N$-fixed point.
  \item  A group $\Gamma$ is said to have \textit{property} 
  $\mathit{(F_B)}$ if $\Gamma > \Gamma $ has 
  relative $(\mathrm{F}_B)$. Equivalently, if 
  for any isometric representation $\rho$ of $\Gamma$ on $B$, 
  $H^1(\Gamma ;B, \rho) =0$ holds.
\end{itemize}
\end{defn}

\subsection{\textbf{Useful lemmas}}
Let $B$ be a superreflexive Banach space, $G \triangleright N $, and 
$F \subset G$ be a compact subset. 
We define the \textit{relative Kazhdan constant for property }$(T_B)$ for 
$(G,N;F, \rho)$ by the following equality: $\mathcal{K}(G,N;F,\rho)= \inf_{\xi \in S(B_1) }
      \sup_{s \in F}   \| \rho (s) \xi -\xi  \|  $. 
Here $B_1 =B_{\rho (N)}'$ as in Proposition \ref{prop:decomp}. 
If $G \triangleright N $ have relative $(\mathrm{T}_B)$ and $F$ generates $G$, 
then for any 
isometric representation $\rho$ on $B$, the constant $\mathcal{K}(G,N;F,\rho)$ 
is strictly positive. 

\begin{lem}\label{lem:kazhconst}
Suppose $B$ is us, $G$ is a compactly generated group and 
$F$ is a compact generating set of $G$. 
Let  $\rho$  be any isometric representation of $G$ on 
$B$,  $\xi$ be any vector in $B$ and $\delta_{\xi} := 
\sup_{s \in F} \| \rho (s) \xi -\xi \|$. 
If a pair $G \triangleright N$ has relative 
$(\mathrm{T}_B)$, then there exists a 
$\rho (N)$-invariant vector $\xi_0 \in B$ with 
$
\| \xi - \xi_0 \| \leq 2 \mathcal{K}^{-1} \delta_{\xi} 
$. 
Here $\mathcal{K}$ stands for the relative Kazhdan constant $\mathcal{K}(G,N;F,\rho) $ for $(\mathrm{T}_B)$.
\end{lem}

\begin{proof}
Decompose $B$ as $B= B_0 \oplus B_1 =B^{\rho (N)} \oplus B'_{\rho (N)}$, and 
 $\xi$  as $\xi=\xi_0 +\xi_1 $ $(\xi_0 \in B_0 , \xi_1 \in B_1)$. Then 
 $  \rho (s) \xi -\xi = (\rho (s) \xi_0 -\xi_0 )+(\rho (s) \xi_1 -\xi_1 )$ is the 
decomposition of $\rho (s) \xi -\xi$.  For a general 
 decomposition $\eta=\eta_0 +\eta_1$, one has $\| \eta_1 \| \leq \| \eta \| + \| \eta_0 \| \leq 2 \| \eta \|$ by applying Proposition \ref{prop:decomp}. 
 Hence  the inequality 
$
2 \delta_{\xi} =2 \sup_{s \in F}\| \rho (s) \xi -\xi \| \geq \sup_{s \in F}\| \rho (s) \xi_1 -\xi_1 \| \geq \mathcal{K} \|\xi_1 \|
$ holds.
\end{proof}

The following lemma and its corollary are well-known, and also important.

\begin{lem}$($lemma of the Chebyshev center$)$ \label{lem:cheb}
Let $B$ be a uc Banach space and $X$ be a $\mathrm{bounded}$ subset. Then 
there exists a unique closed ball with the minimum radius which contains $X$. 
We define the $\mathrm{Chebyshev}$ $\mathrm{center}$ of $X$ as the center of this ball.
\end{lem}

\begin{cor}\label{cor:cocy}
Let $B$ be a superreflexive Banach space and $N$ be a subgroup of $G$. Then for any affine isometric action of $G$ on $B$, the following are equivalent:
\begin{enumerate}[$\mathrm{(}$i$\mathrm{)}$]
  \item  The action has an $N$-fixed point.
  \item  Some $($or equivalently, any$)$ $N$-orbit is bounded.
\end{enumerate}
\end{cor}

\subsection{Elementary linear groups and unit elementary matrices}\label{subsec:semid}
Let $A$ be a unital ring and $n \geq 2$. 
Let $i, j$ be indices with $1\leq i \leq n$, $1\leq j \leq n$, and 
$i \ne j$. For $a\in A$, we let $E_{i , j} (a)$ denote the element in the matrix ring 
$M_n (A)$ whose all diagonal entries are $1$, $(i,j)$-th entry is $a$ 
and the other entries are $0$. In the setting above, we recall that an element in $M_n(A)$ is called an \textit{elementary matrix} if it is of the form $E_{i,j}(a)$ as in above, and that the \textit{elementary linear group} $EL_n(A)$ is defined as the multiplicative group in $M_n(A)$ generated by all elementary matrices. Note that it is also common to call $EL_n(A)$ the \textit{elementary group} over $A$ and use the symbol $\mathrm{E}_n(A)$.

Let $G= EL_{n}(A) \ltimes A^{n}$$\triangleright A^{n}=N$. Then we 
identify $G$ as 
$$G \cong \left\{ (R,v):=
 \left( 
\begin{array}{c|c}
R & v \\
\hline 
 0 &  1 
\end{array}
\right)
 : R \in EL_{n} (A), \, v \in A^{n}
 \right\}
\subset EL_{n+1} (A).$$ 
We also identify $N$ with the additive group of all column vector $v$. 
Here we abbreviate $(I,v) \in N \subset G$ by omitting $I(=I_{n})$.

In the case of $A=A_k=$$\mathbb{Z}[x_1 , \ldots ,x_k]$, we define 
the \textit{unit elementary matrices} as the matrices of the form 
$E_{i,j}(\pm x_l)$ $(0\leq l \leq k)$. Here we set $x_0=1$. 
We also consider the case of that $G= EL_{n}(A_k) \ltimes A_k^{n}$$\triangleright 
A_k^{n}=N$. 
In this case, we define 
the finite generating set $F$ as follows: with the above identification 
$G \subset EL_{n+1} (A_k)$, we let 
$F$ be the set of all unit elementary matrices in $G$. We also 
let $F_1$ denote $F \cap N$ and $F_2$ denote $F \setminus F_1$.

\section{\textbf{Proof of Theorem \ref{thm:TtoF}}}\label{sec:TtoF}
We keep the same notation and identifications as in Subsection 
\ref{subsec:semid} (with $n=3$). We let $N_1$ be the subgroup of 
$N(\subset SL_4(A_k))$ of all elements whose $(2,4)$-th and $(3,4)$-th 
entries are $0$. 
Take an arbitrary 
affine isometric action $\alpha$ on $B$, and choose and fix one norm on $B$ as in Remark \ref{rem:superref}. We decompose $\alpha$ into the linear 
part $\rho$ and the cocycle part $c$. We also decompose 
$B$ as $B=$$B_0 \oplus B_1$ ($B_0=B^{\rho (N)}$ and $B_1= B^{'}_{\rho (N)}$) and 
obtain the associated decomposition $c=c_0 +c_1$. From the $\rho (G)$-
invariance of $B_0$ and $B_1$, each $c_j$, $j \in \{ 0,1\}$ is a $\rho$-cocycle. 
For any elements $g=(R,0)\in G$ and 
$h=v \in N$,  $g h g^{-1} =$$(I ,Rv)=:$$Rv \in N$ holds. 
In particular, by noting that $\rho \mid_N =\mathrm{id}$ on $B_0$, we have the following 
equality: for any $R \in SL_3 (A_k)$ and $v \in N$, 
$c_0 (Rv) = \rho ((R,0)) c_0(v)$. 
Then one can check the following two facts from the equality above 
and the cocycle identity for $c_1$: 
\begin{itemize}
   \item   The set $c_0 (N)$ is bounded (and hence actually equal to $0$).
   \item  If $c_1 (N_1)$ is bounded, then $c_1(N)$ is bounded.
\end{itemize}
(The first part follows from the fact that any column vector in $N$ can be written as a sum of three columns such that for any column of the three, at least one entry is $1$.)

\begin{proof}(\textit{Theorem \ref{thm:TtoF}})
Thanks to the two facts above and Corollary \ref{cor:cocy}, for the proof 
it suffices to verify 
the boundedness of $c_1 (N_1)$. 
We define a finite subset $F_0$ and two subgroups $G_1$, $G_2$  
of $G$ by the following expressions respectively: 
$$
 \left\{  \left(
\begin{array}{cccc}
1 & * & * & * \\
0 & 1 & * & * \\
0 & * & 1 & * \\
0 & 0 & 0 & 1
\end{array}
\right) \right\},\ 
\left\{  \left(
\begin{array}{ccc}
1 & ^t v' & 0 \\
0 &  R'   & 0 \\
0 &  0    & 1 
\end{array}
\right) \right\} \mathrm{\ \ and\ \ } 
\left\{  \left(
\begin{array}{ccc}
1 &  0 & 0 \\
0 &  R' & v' \\
0 &  0  & 1 
\end{array}
\right) \right\}.
$$
Here in the first definition, the expression means that for each 
element in $F_0$, 
only one of the above $*$'s is 
$ \pm x_l $ $(0\leq l \leq k)$ and the others are 0. Also in the second and the 
third expressions, $ R' $ moves among all elements in $ EL_2(A_k)$ and $v'$ 
moves among all elements in $ A_k^2 $. We let 
$C = \sup_{s \in F_0} \| c_1(s) \|$. We set $L \ (\triangleleft G_1)$ as the 
group of all elements in $G_1$ with $R'=I$ and $N_2 \ (\triangleleft G_2)$ as the group 
of all elements in $G_2$ with $R'=I$ . 
A crucial point here is that $N_1$ 
commutes with $F_0$: therefore for any 
$h \in N_1$ and any $s \in F_0$, we have the following inequality: 
\begin{align*}
&\| \rho (s) c_1(h) - c_1(h) \| = \| c_1(sh) - c_1(h) -c_1(s) \| \leq \| c_1 (hs)  - c_1(h) \| + \| c_1(s) \| \\
=&  \| \rho (h) c_1(s) \| +\|c_1(s) \| = 2 \| c_1(s) \| \leq 2C .
\end{align*}

We set a number 
$\mathcal{K}$ as the minimum of the two relative Kazhdan constants $\mathcal{K}(G_1,L;F_0 \cap G_1,\rho \mid_{G_1}) $ and $ \mathcal{K}(G_2,N_2;F_0 \cap G_2 ,\rho \mid_{G_2})$. 
Then from relative $(\mathrm{T}_B)$ of $EL_2(A_k) \ltimes A_k ^2$
$\triangleright A_k ^2$, $\mathcal{K}$ is strictly positive. 
Hence from Lemma  \ref{lem:kazhconst}, for any $\xi \in c_1 (N_1)$ one can choose 
a $\rho (L)$-invariant vector $\eta$ and a $\rho (N_2)$-invariant vector $\zeta$ 
with 
\begin{align*}
&\| \xi- \eta\| \leq 4 \mathcal{K}^{-1}D, \\
\mathrm{and} \ \ \ &\| \xi- \zeta\| \leq 4 \mathcal{K}^{-1}D. 
\end{align*}
Finally, note that $N_1$ is obtained by single commutators between $L$ and $N_2$: 
for any $h \in N$, there exist $h_1 \in N_1$, $h_2 \in N_2$, $h' \in N_2$, and 
$l \in L$ such that $h=h_1 h_2$ and $h_1 =l h' l^{-1} {h'} ^{-1}$. 
Hence for any $\xi \in c_1 (N_1)$ and $h \in N$, the following inequality holds: 
$$
 \| \rho (h) \xi- \xi \| = \| \rho (lh' l^{-1} {h'} ^{-1}h_2) \xi-\xi \| 
\leq  4 \| \xi-\eta \| + 4 \| \xi-\zeta \|  \leq 32 \mathcal{K}^{-1}C.
$$
Note that the upper bound of the inequality above is independent of the 
choices of $\xi \in c_1 (N_1)$ and $h\in N$.

Now suppose that $c_1 (N_1)$ is not bounded. 
Then one can choose 
 $\xi \in c_1 (N_1)$ such that 
  $\| \rho (h)\xi - \xi \| < \|\xi \| $ holds for 
all $h \in N$. 
Then from 
Lemma \ref{lem:cheb}, 
there must exist a \textit{non-zero} $\rho (N)$-invariant 
vector in $B_1=$$B^{'}_{\rho (N)}$, 
but it is a contradiction. 
\end{proof}

\section{\textbf{Proof of Theorem \ref{thm:Tsemi}}}\label{sec:relT}
For the proof, we will concentrate on investigation for the case of 
relative $(\mathrm{T}_{[\mathcal{H}]})$. Indeed, the case of relative 
$(\mathrm{T}_{L^p})$ directly follows from the 
original relative property (T) proved by Shalom 
\cite{Shal1} and the relative version of $(ii)$ in Theorem \ref{thm:BFGM}. 
We keep the same notation and identifications as in Subsection 
\ref{subsec:semid} (with $n=2$).

For any $B \in [\mathcal{H}]$ and any isometric representation $\rho$ on $B$, 
one can regard $\rho$ as a uniformly bounded representation on a Hilbert 
space $\mathcal{H}$. The key to proving Theorem \ref{thm:Tsemi} is the 
following proposition by J. Dixmier \cite{Dix}, that states any uniformly bounded 
representation on a Hilbert space of an \textit{amenable} group is 
unitarizable.
\begin{prop}\label{prop:Dix}$($\cite{Dix}$)$
Let $\Lambda$ be a locally compact group. Suppose $\Lambda$ is $\mathrm{amenable}$. 
Then for any uniformly bounded representation $\rho$ on $\mathcal{H}$ of 
$\Lambda$, there exists an invertible operator $T\in \mathbb{B}(\mathcal{H})$ 
such that $\pi = \mathrm{Ad} (T) \circ \rho =T \circ \rho \circ T^{-1}$ is a unitary representation. Moreover, one can choose $T$ with 
$\| T\|_{\mathbb{B}(\mathcal{H})} \| T^{-1} \|_{\mathbb{B}(\mathcal{H})} \leq  |\rho |^2 $.
\end{prop}

\begin{proof}(\textit{Theorem \ref{thm:Tsemi}}, Outlined)
For simplicity, we shall show the case of $k=0$. Namely, we will prove relative property 
$(\mathrm{T}_{[\mathcal{H}]}$) for 
$N=\mathbb{Z}^2 \triangleleft SL_2 (\mathbb{Z}) \ltimes \mathbb{Z}^2 =G$.

Suppose that there exist a ucus Banach space $B \in [\mathcal{H}]$ and an isometric representation $(\rho , B)$ of $G$ such that $\rho$ admits almost invariant vectors in 
$B_{\rho (N)}'$. We may assume that $B_{\rho (N)}'=B$ 
because $B_{\rho (N)}' $ is also an element in $
[\mathcal{H}]$. 
Thanks to the amenability of $N$ and Proposition \ref{prop:Dix}, 
we may also assume $(\rho , \mathcal{H})$ is a \textit{unitary} representation on 
$N$ . 
We choose any vector $\xi \in S(B)$ and fix it. We let 
$\delta_{\xi} = \sup_{s\in F}\| \rho (s) \xi -\xi\|_B$ and 
${\delta^*}_{\xi} = \sup_{s\in F}\| \rho^{\dagger} (s) \xi^* -\xi^*\|_{B^*}$. 
Here $\xi \mapsto \xi^*$ is the duality mapping defined in Lemma \ref{lem:dual}.

Then from chosen vector $\xi$ and the duality on $B$, we can construct a 
spectral measure $\mu=\mu_{\xi}$ on the Pontrjagin dual $\widehat{\mathbb{Z}^2} \cong $$\mathbb{T}^2 \cong$$\left[  -\frac{1}{2},\frac{1}{2}\right) ^2$. 
The method for constructing the measure is similar to one in the 
original relative (T) 
argument as in \cite{Shal1} or in its slightly different interpretation in a book of N. Brown and N. Ozawa 
\cite[Theorem 12.1.10]{BOz}. 
Unlike the original case of the proof ofrelative $(\mathrm{T})$, 
$\mu$ is complex-valued in general. 
However, one obtains the positive part $\mu_{+}$ by taking the Hahn--Jordan decomposition of $\mu$. 
One can also verify the following three facts by an argument similar to one in the original proof for relative (T) (we refer to Subsection~\ref{subsec:semid} for the definition of the finite subset $F_2$): 
\begin{itemize}
   \item  The inequality $\mu_+ (\mathbb{T}^2) \geq 1$ holds. 
   \item  For any Borel set $W$ being far from the origin $0$ of $\mathbb{T}^2$ 
   (in certain quantitative sense), 
   $\mu_+ (W)= O(\delta_{\xi} \cdot {\delta^*}_{\xi})$  as  
   $\delta_{\xi} ,\, {\delta^*}_{\xi} \to 0$.
   \item  For any Borel subset $Z \subset \mathbb{T}^2$ and 
   $g \in F_2 $, 
   $ |\mu_+ (\hat{g} Z)- \mu_+ (Z) |= O(\delta_{\xi} + {\delta^*}_{\xi})$ as 
    $\delta_{\xi} ,\, {\delta^*}_{\xi} \to 0$. Here for $g \in SL_2 (\mathbb{Z})$, 
    $\hat{g}=(\,^{t}g)^{-1}$ and $SL_2 (\mathbb{Z})$ naturally acts on $\mathbb{T}^2$.
\end{itemize}

Now let $\xi \in S(B)$ move among almost invariant vectors with 
$\delta_{\xi} \to 0$. Then from (uniform) continuity of the duality 
mapping, ${\delta^*}_{\xi}$ also tends to $ 0 $. Hence there must exist some vector 
$\xi \in S(B)$ such that the associated positive measure $\mu_+$ has a non-zero 
value on $\{0 \} \subset\mathbb{T}^2$. This contradicts our assumption that 
$B^{\rho (N)}=0$. 
\end{proof}
We refer to the Appendix for details and a certain quantitative 
treatment.

\section{\textbf{Reduced cohomology, ultralimit, and Shalom's machinery}}\label{sec:FSLnZ}
Throughout this section, we let $\Gamma$ be a \textit{discrete and  finitely 
generated} group and $F$ be a finite generating subset of $\Gamma$. 
Shalom \cite{Shal2} has defined the following property: an affine isometric action $\alpha$ of 
$\Gamma$ on a Banach space $B$ 
is said to be \textit{uniform} if 
there exists $\varepsilon >0$ such that 
$\inf_{\xi \in B} \sup_{s \in F} $$\| \alpha (s)\xi -\xi \| $$\geq \varepsilon$ 
holds. 
We note whether an action is uniform is determined independently of the choice of finite generating 
set $F$.  The conception of uniformity of actions is closely related to the 
closure $\overline{B} ^1 (\Gamma;B, \rho)$ of the coboundary 
$B^1 (\Gamma;B, \rho)$. More precisely,  for any isometric 
representation $\rho$, a $\rho$-cocycle $c$ is in 
$\overline{B} ^1 (\Gamma;B, \rho)$ if and only if the associated affine action 
$\alpha = \rho +c$ is \textit{not} uniform. 
\begin{defn}\label{def:reduced cohom}
The \textit{reduced first cohomology} of $\Gamma$ with $\rho$-coefficient 
is defined as the additive group 
$\overline{H} ^1 (\Gamma ;B, \rho)$$=Z^1(\Gamma ;B, \rho) / \overline{B} ^1(\Gamma ;B, \rho)$.
\end{defn}
In \cite[Theorem 6.1]{Shal2}, Shalom has shown the following theorem: ``
\textit{Suppose }$G$\textit{ is a}
 compactly generated \textit{topological group. If} $G$ 
 \textit{fails to have }(FH)\textit{, 
 then there exists a unitary representation }$(\pi ,\mathcal{H})$\textit{ with }
 $\overline{H} ^1 (G ;\mathcal{H}, \pi) \ne 0$." 
At least in the case of discrete groups, one can extend this theorem 
to more general situations. One extension was essentially found by 
 Gromov \cite{Gr}, and his idea is to take a \textit{scaling limit}.

An \textit{ultralimit} means a unital, positive and multiplicative 
$*$-homomorphism $\omega $-$\lim :\ell^{\infty}(\mathbb{N})\to \mathbb{C}$ such 
that for any $\left( \xi_n \right) _{n=0}^{\infty}$ converging to some 
element, $\omega $-
$\lim \left( \xi_n \right) = \displaystyle{\lim_{n \to \infty} \xi_n}$ holds. 
Choose any ultralimit $\omega $-$ \lim$ and fix. Then one can define 
 the \textit{ultralimit of Banach spaces} 
 $\left( B_{\omega}, \| \cdot \|_{\omega} , \zeta_{\omega} \right)$ 
 for any sequence $\left( B_n, \| \cdot \|_n , \zeta_n  \right)_n$ 
 of (affine) Banach spaces, norms and base points. 
Moreover, let $(\alpha_n , B_n)_n$ be a sequence of affine isometric actions 
of $\Gamma$. If the condition 
 $\sup_{s\in F} \sup_{n} \| \alpha_n (s)\zeta_n -\zeta_n \| < + \infty $ holds, 
then we can naturally define the \textit{ultralimit of actions} 
$\alpha_{\omega}$ on $B_{\omega}$. We refer to Silberman's website \cite{Sil}
 for details of above and for a proof of the following proposition. 
\begin{prop}$($proposition of scaling limit$)$\label{prop:scallim}
Let $\alpha$ be an affine isometric action of $\Gamma$ on a Banach space $B$. 
Suppose $\alpha$ is not uniform but has no fixed point. Then there exist a 
sequence of base points and positive numbers $(\zeta_n, b_n)$  with 
$\lim_{n} b_n = +\infty$ such that the ultralimit action $\alpha_{\omega}$ 
on $B_{\omega}= \omega$-$\lim \left( B, b_n \| \cdot \|, \zeta_n \right)$ is 
uniform.
\end{prop}

\begin{cor}\label{cor:reduced}
Let $\mathcal{C}$ be a class of Banach spaces which is stable under 
ultralimits. If a finitely generated discrete group $\Gamma$ does not have 
property $(\mathrm{F}_{\mathcal{C}})$, then there exist $B \in \mathcal{C}$ 
and an isometric representation $\rho$ of $\Gamma$ on $B$ such that 
$\overline{H} ^1(\Gamma ;B, \rho) \ne 0$. 
\end{cor}

We extend the conception of the \textit{Shalom property}, which 
is found in \cite[Definition 12.1.13]{BOz}. Before  defining this, we recall the notion of \textit{bounded generation}. In this paper, subsets $(S_j )j\in J$ of a group $G$ ($J$ is in some index set and usually we assume some $S_j$ contains the identity $e_G$) is said to \textit{boundedly generate} $G$ if there exists $l\in \mathbb{N}$ such that $G= (\bigcup_{j\in J}S_j)^{l}$ holds. Namely, any $g\in G$ can be written as a product of $l$ elements where each of them is respectively an element in some corresponding $S_j$. (Note: keep in mind that in some other literature, the terminology ``bounded generation" is used for the following \textit{confined} situation: $J$ is a finite set, and all $S_j$'s are cyclic groups.)

\begin{defn}\label{def:shalom}
Let $B$ be a Banach space and $\Gamma$ be a finitely generated 
group. A triple of subgroups $(G, H_1, H_2)$ of $\Gamma$ is said to have the 
\textit{Shalom property for }$(F_{B})$ if all of the following four conditions hold:
\begin{enumerate}[$\mathrm{(}$1$\mathrm{)}$]
   \item The group $\Gamma$ is generated by $H_1$ and $H_2$ together.
   \item The subgroup $G$ normalizes $H_1$ and $H_2$.
   \item The group $\Gamma$ is boundedly generated 
   by $G, H_1$, and $H_2$. 
   \item For both $i\in \{ 1,2 \}$, $H_i < \Gamma$ has relative $(\mathrm{F}_{B})$.
\end{enumerate}
\end{defn}

Now we shall introduce Shalom's machinery. We refer to \cite[\S 4]{Shal3} for 
the original idea for the case of that $\mathcal{C}=\mathcal{H}$. 

\begin{thm}$($Shalom's Machinery$)$\label{thm:Shalom}
Let $\mathcal{C}$ be a class of superreflexive Banach spaces which is stable 
under ultralimits. Let $\Gamma$ be a finitely generated group with  
finite abelianization. Suppose there exist subgroups $G, H_1$, and $H_2$ of 
$\Gamma$ such that $(G, H_1 , H_2)$ has the Shalom property for 
$(\mathrm{F}_{\mathcal{C}})$.  Then $\Gamma$ has property 
$(\mathrm{F}_{\mathcal{C}})$.
\end{thm}

\begin{proof}
Suppose the contrary. Then from Corollary \ref{cor:reduced}, there must exist 
 an affine isometric action $\alpha_0$ on some $B_0 \in \mathcal{C}$ such 
that $\alpha_0$ is uniform. For simplicity, we may assume that $B_0$ is uc. 
Fix a finite generating set $F$ of $\Gamma$. We set $\mathcal{A}$ as the 
class of all pairs $(\alpha , E)$ of an affine isometric action and a uc 
Banach space with the following two conditions: first, 
 for any $\xi\in E$, $\sum_{s\in F} \| \alpha (s)\xi-\xi\|_E \geq 1$ 
holds. Second, for all $0<\varepsilon <2$, the value 
of the modulus of convexity of $E$ at $\varepsilon$ is not less than that of 
$B_0$. (We refer to Definition \ref{def:modulus}.) We note that this class 
$\mathcal{A}$ is non-empty. Furthermore, thanks to 
\cite[\S 2, Theorem 4.4]{AK}, $\mathcal{A}$ is stable under ultralimits. 

Next we define a real number $D$ as 
$\inf \{ \|\xi^1 -\xi^2 \| : (\alpha , E)\in \mathcal{A} \}$. Here for 
$i\in \{ 1,2\}$, $\xi^i$ moves through all $\alpha (H_i)$-fixed points in $E$. We 
check that the definition above makes sense (, namely, this set is nonempty and hence this infimum is finite) with the aid of condition $(4)$. By taking an ultralimit, one can show that $D$ is actually a minimum. 
Let $\xi^1_{\infty}$ and $\xi^2_{\infty}$ be vectors which attain the minimum $D$. 
Also let $(\alpha_{\infty}, E_{\infty}) \in \mathcal{A}$ be the associated 
affine action and $\rho_{\infty}$ be the linear representation for $\alpha_{\infty}$. 

Decompose the action $\alpha_{\infty}$ into $\alpha_{\infty}^{\mathrm{triv}}$ and $\alpha_{\infty}^{'}$, where the former takes values in $E_{\infty}^{\rho_{\infty} (\Gamma)}$ and the latter tales values in $E_{\infty ,\ \rho_{\infty} (\Gamma)} ^{'}$. Then from the strict convexity of $E_{\infty}$ and condition $(2)$, one can conclude that the $\alpha_{\infty}^{'} (G)$-orbit of $E_{\infty, \rho_{\infty}(\Gamma)}^{'}$ component of each $\xi^i_{\infty}$, $i\in \{ 1,2 \}$ is one point (otherwise $E_{\infty ,\, \rho_{\infty} (\Gamma)} ^{'}$ must contain a non-zero $\rho_{\infty}(\Gamma)$-invariant vector and contradiction occurs), and hence bounded. Note that $\alpha_{\infty}^{\mathrm{triv}} (G)$-orbits of the $E_{\infty}^{\rho (\Gamma)}$ components are also bounded because $\Gamma$ has  finite abelianization. Therefore in particular, 
$\alpha_{\infty} (\Gamma )$-orbit of $\xi^1_{\infty}$ must be bounded with the use of condition (3). 
However this exactly means $\alpha_{\infty}$ has a global fixed point, and it contradicts the definition of $\mathcal{A}$.
\end{proof}

\begin{proof}(\textit{Theorem \ref{thm:FSLnZ}})
Let $\Gamma = SL_n(A_k)$, $G \cong SL_{n-1}(A_k)$, and 
$H_1 ,H_2 \cong {A_k}^{n-1}$. Here in $\Gamma$ we realize $G$ as in the left upper corner (, namely, the 
$((1$-$(n-1))\times(1$-$(n-1)))$-th parts), realize $H_1$ as in the 
$((1$-$(n-1))\times n) $-th unipotent parts, and realize $H_2$ as in the 
$(n \times(1$-$(n-1))) $-th unipotent parts. 
One can directly check that $\Gamma$ has the trivial abelianization from commutator relations among elementary matrices. 
We claim that 
$(G, H_1, H_2)$ has the Shalom property for $(\mathrm{F}_{\mathcal{C}})$. Indeed, 
conditions $(1)$ and $(2)$ are confirmed directly, and condition $(4)$ 
follows from Theorem \ref{thm:TtoF} and Theorem \ref{thm:Tsemi}. For condition 
$(3)$, the following deep theorem of Vaserstein \cite{Vas} insures the assertion: ``\textit{For any }$n\geq 3$\textit{, the group }$SL_n(A_k)$\textit{ is bounded generated by} $SL_2(A_k)$ $($\textit{in the upper left corner}$)$ \textit{and all elementary matrices.}"

Thanks to Theorem \ref{thm:Shalom}, for our proof it only remains to show that $\mathcal{C}$ 
is stable under  ultralimits. In the case of that $\mathcal{C}=L^p$, it 
follows from a work of S. Heinrich \cite{He} (we also refer to \cite[\S 15.Theorem 3]{Lac} and \cite[\S 2]{AK}). In the case of that $\mathcal{C}=[\mathcal{H}]$, it 
 is not stable. However for any $M \geq 1$, the following class 
$\mathcal{B}_M$ is stable under ultralimits: 
we define $\mathcal{B}_M$ as the 
class of all elements $B$ in $[\mathcal{H}]$ which have 
compatible Hilbert norms with the norm ratio $\leq M$. 
By noticing that 
$[\mathcal{H}]= \bigcup_{M \geq 1} \mathcal{B}_M$, one accomplishes the 
proof.
\end{proof}

The author does not know whether the assertion of Theorem \ref{thm:FSLnZ} 
is satisfied for the noncommutative universal lattice 
$EL_n$$ (\mathbb{Z} \langle x_1 , x_2 , \cdots , x_k \rangle)$ $(n \geq 3)$. In the case above with $n\geq 4$, although most of the ingredients are still valid, 
the bounded generation property may fail. We note 
that nevertheless, M. Ershov and A. Jaikin-Zapirain \cite{EZ} have proved 
property $(\mathrm{T})$ for noncommutative universal lattices for $n\geq 3$.

\section{\textbf{Quasi-actions and property $(\mathrm{FF}_B)$}}\label{sec:FFB}
\begin{defn}\label{def:quasi}
Let $B$ be a Banach space and $\Gamma$ be a group.
\begin{itemize}
   \item  A  map $\beta$ from $\Gamma$ to the set of all affine 
   isometries on $B$ is called a \textit{quasi-action} if the expression 
   $\sup_{g,h \in \Gamma}\sup_{\xi\in B } \| \beta (gh)\xi -\beta(g) \beta (h)\xi\|  $   is finite.
   \item  Let $\rho$ be an isometric representation. 
   A  map $b$ from $\Gamma$ to $B$ is called a 
   \textit{quasi-}$\rho$\textit{-cocycle}  if the expression 
   $\sup_{g,h \in \Gamma} \| b(gh)- b(g) - \rho (g) b (h) \|  $   is finite.
\end{itemize}
\end{defn}
\begin{rem}\label{rem:almstcocy}
In the definition of quasi-actions, one 
can decompose the map $\beta$ 
into the linear part $\rho$ and the transition part $b$, namely, 
$\beta (g)\xi= \rho (g)\xi + b(g)$ for any $g\in \Gamma$ and $\xi\in B$. 
Then from a standard argument, the map $\beta$ is a quasi-action 
if and only if $\rho$ is a group representation 
and $b$ is a quasi-$\rho$-cocycle. 
\end{rem}

Next we define property $(\mathrm{FF}_B)$ as follows. 
We mention that the original terminology in Monod's work \cite{Mon} for 
$(\mathrm{FF}_{\mathcal{H}})$ is 
\textit{property }$(TT)$. We use the 
terminology $(\mathrm{FF}_{B})$ because this property is a quasification of 
$(\mathrm{F}_B)$, \textit{not} of $(\mathrm{T}_B)$. 

\begin{defn}\label{def:ffb}
Let $B$ be a Banach space.
\begin{itemize}
   \item  A pair $G>N$ of groups is said to have \textit{relative property }
   $(FF_B)$ if for any quasi-action on $B$, some (or equivalently, any) 
   $N$-orbit is bounded. This is equivalent to the condition that 
   for any isometric representation $\rho$ of $G$ on $B$ and any 
   quasi-$\rho$-cocycle $b$, $b(N)$ is bounded. 
   \item  
   A group $\Gamma$ is said to have \textit{property} 
   $(FF_B)$ if $\Gamma > \Gamma$ has relative $(\mathrm{FF}_B)$.
   \item  
   Then a group $\Gamma$ is said to have \textit{property} 
   $(FF_B)/T$ (, which is called ``\textit{property} $(FF_B)$ \textit{modulo trivial part}",) if for any isometric representation $\rho$ of $\Gamma$ on $B$ 
   and    any quasi-$\rho$-cocycle $b$, $b'(\Gamma)$ is 
   bounded, where $b'\colon \Gamma \to B/B^{\rho(\Gamma)}$ is the natural quasi-cocycle constructed from the projection of $b$ to $B/B^{\rho(\Gamma)}$. If $B$ is superreflexive, then this definition is equivalent to the following condition: for any isometric representation $\rho$ of $\Gamma$ on $B$ and 
   any quasi-$\rho$-cocycle $b$, $b_1(\Gamma)$ is 
   bounded. Here we decompose $b$ as $b_0+b_1$ such that $b_0$ takes 
   values in $B_0=B^{\rho (\Gamma)}$ and $b_1$ takes values in 
   $B_1=B'_{\rho(\Gamma)}$. 
\end{itemize}
\end{defn}

By observing our proof of Theorem \ref{thm:TtoF}, one can extend 
the argument to the case of that $c$ is a quasi-$\rho$-cocycle. 
Thus one obtains the following theorem. 
\begin{thm}\label{thm:TtoFF}
With the same notation as one in Theorem \ref{thm:TtoF}, let $B$ be any 
superreflexive Banach space. If $EL_2 (A_k)\ltimes A_k ^2 \triangleright A_{k}^2$ has relative $(\mathrm{T}_B)$, then $SL_3 (A_k)\ltimes A_k ^3 > A_{k}^3$ has relative $(\mathrm{FF}_B)$.
\end{thm}

We define the following property to prove Theorem \ref{thm:FFB}. 

\begin{defn}\label{def:shalomFF}
With the same notation as in Definition \ref{def:shalom}, the triple of subgroups $(G , H_1, H_2)$ of 
$\Gamma$ is said to have the \textit{Shalom property for} $(FF_B)$ if the 
following four  
 conditions $(1)$, $(2)$, $(3)$, and $(4')$ hold: 
conditions $(1)$, $(2)$, 
 and $(3)$ are same as in Definition \ref {def:shalom}. And 
 we define a new condition $(4')$ by replacing 
 \textit{relative} $(F_B)$ with \textit{relative} $(FF_B)$ in condition $(4)$ 
 in Definition \ref{def:shalom}. 
\end{defn}

\begin{prop}\label{prop:bddgen}
Let $B$ be a superreflexive Banach space and $\Gamma$ be a  group. 
Suppose $\Gamma$ has property $(\mathrm{T}_B)$ and 
there exist subgroups $G, H_1$, and $H_2$ of \ $\Gamma$ such that 
$(G, H_1 , H_2)$ has the Shalom property for $(\mathrm{FF}_{B})$. 
Then $\Gamma$ has property $(\mathrm{FF}_{B})/\mathrm{T}$. 
\end{prop}

\begin{proof}
For simplicity, we assume that $B$ is ucus. 
Let $\rho$ be an arbitrary isometric representation of $\Gamma$ on $B$ 
and $b$ be an arbitrary quasi-$\rho$-cocycle. 
We decompose 
$B=$$B_0 \oplus B_1$$=B^{\rho (\Gamma)}\oplus B^{'}_{\rho (\Gamma)}$ and 
$b=b_0+b_1 $. 
We set $
\Delta= \sup_{g ,h \in \Gamma} \|  b_1(g) + \rho (g) b_1(h) -b_1(gh) \|
$. 
From condition $(4')$ of the Shalom property for $(\mathrm{FF}_B)$, there 
exists a positive number $C$ such that for any $h \in H_1 \cup H_2$, 
$
\| b_1(h) \| \leq C
$ holds.
By making use of condition $(2)$, one obtains the following inequality: 
for any $g\in G$ and 
$h \in H_1 \cup H_2$, 
\begin{align*}
  &\| \rho (h) b_1 (g) -b_1(g) \| \leq \|  b_1 (hg)  -b_1(g) \| +\|b_1(h)\| +\Delta \\
   \leq & \| b_1 (hg)- b_1(g) \| +\Delta+ C   = \| b_1 (gg^{-1}hg) - b_1(g) \| +\Delta+ C \\
   \leq & \| \rho (g)b_1(g^{-1}hg)\| +2\Delta+ C \leq 2(\Delta+C).
\end{align*} 
Let $S$ be any finite subset of $\Gamma$. 
From the inequality above and condition $(1)$,  
$\sup_{s \in S}\| \rho(s)\xi - \xi \|$ is bounded independently of the choice of 
$\xi \in b_1 (G)$.

Now suppose that $b_1(G)$ is not bounded. 
Then $\rho$ must admit almost invariant vectors in $B_1$, but 
it contradicts property $(\mathrm{T}_B)$ for $\Gamma$.  
Therefore $b_1(G)$ is bounded. 
Finally, one obtains  the boundedness of  
$b_1(\Gamma)$ through use of the bounded generation (condition $(3)$).
\end{proof}

\begin{proof}(\textit{Theorem \ref{thm:FFB}})
The conclusion follows from Theorem \ref{thm:FSLnZ}, Theorem 
\ref{thm:TtoFF}, and Proposition \ref{prop:bddgen}.
\end{proof}

The author does not know whether similar boundedness property holds for \textit{trivial} linear part.

\begin{rem}
Recently, Ozawa \cite{Oz} has strengthened property $(\mathrm{TT})$$(=$ property $(\mathrm{FF}_{\mathcal{H}})$$)$ and defined the concept of \textit{property} $(TTT)$. In \cite{Oz}, he has proved that $EL_2 (A_k) \ltimes A_k ^2 $$\triangleright A_k ^2$ has relative property $(\mathrm{TTT})$, where $A_k$ means $\mathbb{Z}[x_1 , \ldots , x_k]$. Hence from Proposition \ref{prop:bddgen}, it is established that the universal lattice $SL_3 (A_k)$ has $(\mathrm{FF}_{\mathcal{H}})/\mathrm{T}$ (this property is also written as \textit{property} $(TT)/T$).  Ozawa employs theory of the Fock Hilbert spaces and positive definite kernels to obtain the relative property $(\mathrm{TTT})$ in above. Therefore it may remain unknown whether $SL_3 (A_k)$ possesses property $(\mathrm{FF}_{\mathcal{C}})/\mathrm{T}$ or (more weakly,) property $(\mathrm{F}_{\mathcal{C}})$. Here $\mathcal{C}$ stands for $L^p$ $(2<p<\infty)$ or $[\mathcal{H}]$. 
\end{rem}

\section{\textbf{Applications}}\label{sec:appli}
\subsection{\textbf{Actions on the circle}}
Let $\mathrm{S}^1$ be the unit circle in $\mathbb{R}^2$ and identify $\mathrm{S}^1$ with $[-\pi , \pi )$. We denote by $\mathrm{Diff}_{+}(\mathrm{S}^1)$  the group of orientation preserving group diffeomorphisms of $\mathrm{S}^1$. 
\begin{defn}
Let $\alpha > 0$ be a real number. The group $\mathrm{Diff}^{1+\alpha}_{+}(\mathrm{S}^1)$ is defined as the class of all orientation preserving group diffeomorphisms $f$ of $\mathrm{S}^1$ such that $f'$ and $(f^{-1})'$ are H\"{o}lder continuous with exponent $\alpha$. Here a function $g$ on $\mathrm{S}^1$ is said to be \textit{H\"{o}lder continuous with exponent }$\alpha$ if 
$$
\| g \|_{\alpha}= \sup_{\theta_1 \ne \theta_2} \frac{|g(\theta_1)-g(\theta_2) |}{|\theta_1 -\theta_2|^{\alpha}} < \infty  
$$
holds. 
\end{defn}

A. Navas \cite{Nav1} has shown the following theorem: \textit{For any discrete group }$\Gamma$\textit{ with property }$(\mathrm{T})$\textit{, every homomorphism from }$\Gamma$\textit{ into }$\mathrm{Diff}^{1+\alpha}_{+} (\mathrm{S}^{1}) $\textit{ has finite image, for any }$\alpha > 1/2$\textit{.} He has also noted in \cite[Appendix]{Nav2} that his theorem can be extended to general $L^p$ cases. (See also \cite[\S1.b]{BFGM}.)

\begin{thm}\label{thm:Navas}$($\cite{Nav2}$)$
Let $1<p <\infty$ and $\Gamma$ be a discrete group with property $(\mathrm{F}_{L^p})$. Then for any $\alpha > 1/p$,  every homomorphism $\Gamma \to \mathrm{Diff}^{1+\alpha}_{+} (\mathrm{S}^{1}) $ has finite image.
\end{thm}

For the proof, Navas generalizes the argument in \cite{Nav1} by using the \textit{Liouville} $L^p$ \textit{cocycle} of $\mathrm{Diff}^{1+\alpha}_{+}(\mathrm{S}^1)$ on $L^p (\mathrm{S}^1 \times \mathrm{S}^1)$, that is, 
$$
c_p (g^{-1})(\theta_1,\theta_2)=\frac{[g'(\theta_1)g'(\theta_2)]^{1/p}}{\left| 2\sin{\left((g(\theta_1)-g(\theta_2))/2 \right)}  \right|^{2/p}} - \frac{1}{\left| 2 \sin{\left((\theta_1 -\theta_2 )/2 \right)} \right|^{2/p}}. 
$$

\begin{proof}(\textit{Corollary \ref{cor:circle}})
It is straightforward from Theorem \ref{thm:Navas} and Theorem \ref{thm:FSLnZ}.
\end{proof}

\subsection{\textbf{Bounded cohomology}}
We would like to refer to  Monod's book \cite{Mon} for details on bounded cohomology. Throughout this subsection, we let $\Gamma$ be a \textit{discrete} group.
\begin{defn}(\cite{Mon})\label{def:bddcoh}
Let $(B, \rho)$ be a \textit{Banach }$\Gamma$\textit{-module}, namely,  $B$ be a Banach space and $\rho$ be an isometric 
representation of $\Gamma$ on $B$. 
\begin{itemize}
 \item The \textit{bounded cohomology} $ H_b ^{\bullet} (\Gamma; B, \rho)$ of 
 $\Gamma$ with coefficients in $(B ,\rho)$ is defined as the cohomology of the 
 following cochain complex:
  $$
  0 \longrightarrow \ell^{\infty}(\Gamma , B)^{\rho (\Gamma )} 
  \longrightarrow \ell^{\infty}(\Gamma^2 , B)^{\rho (\Gamma )} 
  \longrightarrow \ell^{\infty}(\Gamma^3 , B)^{\rho (\Gamma )} 
  \longrightarrow \cdots
  $$
 \item The \textit{comparison map} is the collection of linear maps 
 $
  \Psi^{\bullet} \colon H_b ^{\bullet}(\Gamma; B, \rho) \to H^{\bullet}(\Gamma ;B, \rho),
 $
 where the maps above are naturally determined by the complex inclusion.
\end{itemize}
\end{defn}
We note that in general the comparison map is \textit{neither} injective \textit{nor} 
surjective. 
\begin{proof}(\textit{Corollary \ref{cor:bddcohom}})
From the same argument as one in 
\cite[Proposition 13.2.5]{Mon}, one can show the following isomorphism among vector spaces: 
$$
\mathrm{Ker}\Psi^2 \cong \{\mathrm{all\ quasi}\textrm{-}\rho\textrm{-}\mathrm{cocycles} \}/(\{\mathrm{all}\ \rho\textrm{-}\mathrm{cocycles} \}+ \{\mathrm{all\ bounded\ map\colon}\Gamma\to B\}).
$$ 
Hence the conclusion follows immediately.
\end{proof}

\appendix

\section{\textbf{Relative Kazhdan constant for property }$(\mathrm{T})$\textbf{ for} \textbf{uniformly bounded representations}}
We define the Kazhdan constant for relative property 
$(\mathrm{T}_{[\mathcal{H}]})$. For quantitative treatments, it is 
more convenient to focus on a Hilbert space $\mathcal{H}$. 
Therefore we define the extension of the Kazhdan constant 
in terms of uniformly bounded representations on $\mathcal{H}$.

\begin{defn}\label{def:extended const}
Let $ \Gamma \triangleright N$ be a pair of groups, and $F$ be a compact 
subset of $\Gamma$. For $M\geq1$, we define $\mathcal{A}_M$ as the class of 
all pairs $(\rho , \mathcal{H})$ with $| \rho | \leq M$. We define the 
\textit{generalized relative Kazhdan constant} 
\textit{for uniformly bounded representations} by 
\begin{equation*}
      \overline{\mathcal{K}}(\Gamma ,N;F;M)= \inf_{(\rho,\mathcal{H}) \in \mathcal{A}_M}      \ \inf_{\xi \in S\left( \mathcal{H}'_{\rho (N)}\right) }      \sup_{s \in F}   \| \rho (s) \xi -\xi  \|_{\mathcal{H}}.
\end{equation*}
\end{defn}

Our proof of Proposition \ref{prop:relconst} is a development of Kassabov's work \cite{Kas}, originally by M. 
Burger \cite{Bur} and Shalom \cite{Shal1}.  We make use of the following 
quantitative version of Lemma \ref{lem:dual}. 

\begin{lem}$($\cite[Proposition A.5]{BL}, modified$)$\label{lem:Bl}
Let $B$ be us. Suppose $0<\kappa <2$. Then for all $\xi,\eta \in S(B)$ 
with $\|\xi- \eta\| \leq \kappa$, the inequality 
$\|\xi^* -\eta^* \|_{*}\leq  2  r_{\| \cdot \|}(2 \kappa) / \kappa$ holds.
\end{lem}

Let $\Gamma$ be a group and $M \geq 1$. For any 
$(\rho , \mathcal{H}) \in \mathcal{A}_M$, we define the norm 
$\| \cdot \| _{\rho}$ on $\mathcal{H}$ as the dual norm of the following norm 
$\| \cdot \|_{\rho*}$: for  $\phi \in \mathcal{H}^{*}$, 
$\| \phi \|_{\rho*} :=$$ \sup_{g \in \Gamma} \| \rho^{\dagger} (g) \phi \|_{\mathcal{H}^*}$. 
This norm $\| \cdot \| _{\rho}$ satisfies the following three properties: 
firstly, $\| \cdot \| _{\rho}$ is compatible with $\| \cdot \| _{\mathcal{H}}$ 
with the norm ratio $\leq M$. 
Secondly, $\rho$ is isometric with respect to $\| \cdot \| _{\rho}$. Thirdly, 
$(\mathcal{H},$ $\| \cdot \| _{\rho})$ is us. Indeed, thanks to Lemma \ref{lem:lindenst} one has that for any 
$\tau >0$, the inequality $r_{\| \cdot \|_{\rho}}(\tau) \leq \sqrt{1+M^2 \tau ^{2}} -1$
$\leq M^2 \tau ^2 /2$ holds. 
 
\begin{proof}(\textit{Proposition \ref{prop:relconst}})
We stick to the notation and the identifications in Subsection \ref{subsec:semid} (with $n=2$). 
Let $\varepsilon >0$. 
Suppose that there exists $(\rho , \mathcal {H}) \in \mathcal{A}_M$ 
such that $\rho$ admits a non-zero vector $\xi$ in $\mathcal{H}_{\rho(N)}^{'}$ 
which satisfies 
$\sup_{s\in F} \| \rho (s) \xi-\xi \|_{\mathcal{H}}  \leq  \varepsilon \|\xi \|_{\mathcal{H}}$. We may assume that $\mathcal{H}^{\rho(N)}=0$. 
For this $(\rho , \mathcal{H})$, we take the us norm 
$\| \cdot \|_{\rho}$ defined in the paragraph above. Thus by applying Lemma 
\ref{lem:Bl}, we can assume that 
there exists $\xi \in \mathcal{H}$ with $\| \xi \|_{\rho}=1$ such that 
$$ \sup_{s\in F} \| \rho (s) \xi-\xi \|_{\rho}  \leq M \varepsilon \mathrm{\ \ and\ \ }  \sup_{s\in F} \| \rho^{\dagger} (s) \xi^*-\xi^* \|_{\rho*}  \leq 4M^3 \varepsilon.
$$

Thanks to Dixmier's unitarization, we have an invertible operator 
$T\in \mathbb{B}(\mathcal{H})$ with 
$\| T  \|_{\mathbb{B}(\mathcal{H})} \| T^{-1} \|_{\mathbb{B}(\mathcal{H})} \leq M^2$ 
such that $ \pi := \mathrm{Ad}(T) \circ \rho \mid_{N} $ 
is unitary. Let $\hat{N}$ denote the Pontrjagin dual of $N$. 
By general theory of Fourier analysis, 
one  obtains a standard unital $*$-hom 
$\sigma : \mathrm{C}(\hat{N}) \to \mathbb{B}(\mathcal{H})$ from 
the unitary operators $\pi (N)$. 
Indeed, for $i \in \{1,2 \}$, let $z_i \in \mathrm{C}(\mathbb{T}^2)$ be the map  $t \mapsto e^{2\pi \sqrt{-1} t_i}$ ($t_i$ is the $i$-th component of $t\in \mathbb{T}^2$) and let $h_i \in F_1$ be $E_{i,3}(1)$ as in Subsection \ref{subsec:semid}. And we define $\sigma$ by setting $\sigma (z_i)=T\rho(h_i)T^{-1}$ for each $i$. 

Then from Riesz--Markov--Kakutani theorem, from this $\sigma$ 
one obtains 
the complexed-valued regular Borel measure 
$\mu $ on $\hat{N}$ satisfying the following: 
for any $f \in \mathrm{C}(\hat{N})$, 
$\int_{\hat{N}}f\,d\mu =  \left\langle T^{-1} \sigma (f) T \xi, {\xi}^* \right\rangle$. 
(We note that $T=I$ in our proof of Theorem \ref{thm:Tsemi}.) 
We take the Jordan decomposition of $\mathrm{Re}\mu=\mu_+ - \mu_-$. 
Here $\mu_+ \perp \mu_-$ (this means they are singular to each other) and both of them are positive regular Borel measure. 
Then the inequality $\mu_+ (\hat{N}) \geq 1$ holds.

For the proof of Proposition  \ref{prop:relconst}, first we discuss the case of that $k=0$. We take the following well-known decomposition of $\hat{N}=\mathbb{T}^2 \cong  \Bigl\{ \begin{scriptsize}
   \Bigl(
  \begin{array}{cc}
  t_1 \\
  t_2
  \end{array}
  \Bigr) \end{scriptsize}
  :t_1,\, t_2 \in \left[ -\frac{1}{2} , \frac{1}{2} \right) \Bigr\}$:
\begin{align*}
&\{ 0 \}, \ D_0 =\{ |t_1| \geq 1/4 \mathrm{\ or\ } |t_2| \geq 1/4 \},\\
&D_1 =\{ |t_2|\leq |t_1|<1/4  \mathrm{\ and\ } t_1 t_2 >0 \}, 
\ D_2 =\{ |t_1|< |t_2|<1/4  \mathrm{\ and\ } t_1 t_2 \geq 0 \}, \\
&D_3 =\{ |t_1|\leq |t_2|<1/4  \mathrm{\ and\ } t_1 t_2 <0 \}, 
\ D_4 =\{ |t_2|< |t_1|<1/4  \mathrm{\ and\ } t_1 t_2 \leq 0 \}.
\end{align*}
We consider the natural $SL_2(\mathbb{Z})$-action on $\mathbb{T}^2$ defined 
as follows: for any $g \in SL_2(\mathbb{Z})$, the action map $\hat{g}$ of $g$ 
$\colon t \mapsto \hat{g}t$ is the left multiplication of the matrix 
$ \hat{g} = (\,^{t}g)^{-1}$. This action naturally induces 
the $SL_2(\mathbb{Z})$-action on $\mathrm{C}(\mathbb{T}^2)$ as 
$\hat{g}f(t)$$=f(\hat{g}t)$. 
Then one can check the following equality: 
for any $g\in SL(2,\mathbb{Z})$ and any $f \in \mathrm{C}(\mathbb{T}^2)$, 
$
 \sigma (\hat{g} f)= T \rho (g) T^{-1} \sigma (f)  T \rho (g^{-1}) T^{-1}
$.  
With some calculation, one can also obtain the following two 
estimations: 
\begin{itemize}
  \item The inequality $\mu_+ (D_0) \leq 4M^7 \varepsilon^2$ holds.
  \item For any Borel subset $Z \subset \mathbb{T}^2$ and any 
  $g \in  F_2 (\subset SL_2(\mathbb{Z}))$ as in Subsection \ref{subsec:semid}, the inequality 
  $| \mu_+ (\hat{g}  Z)- \mu_+ (Z) | \leq 5M^6 \varepsilon$ holds.
\end{itemize}
Indeed, for instance, the former inequality follows from the argument below. For $i=1,2$, set $D_0 ^{i} = \mathrm{supp} \mu_+ \cap \{ |t_1| \geq 1/4 \} $$\subset D_0$. By approximating (pointwisely) $\chi_{D_0 ^{i}}$ by continuous functions and obtaining an associated projection $P \in \mathbb{B}(\mathcal{H})$, one can make estimate as follows: for each $i \in \{ 1,2\}$, 
\begin{align*}
& 2 \mu_+(D_0^{i})  \leq \left| \int_{D_0 ^{i}}\left | 1- z_i \right|^2  \, d\mu  \right|= \left|\int_{\mathbb{T}^2} \overline{\left( 1- z_i \right)} \chi_{D_0 ^{i}} ( 1- z_i )  \, d\mu  \right| \\
=&\left| \left\langle T^{-1} \sigma\left( 1-z_i \right)^* P \,\sigma\left(1-z_i\right) T\xi , \xi^* \right\rangle \right|  \\
=&\left| \left\langle T^{-1} \left(I- T\rho(h_i) T^{-1}\right)^* P  \left(I- T \rho(h_i) T^{-1}\right)T \xi , \xi^* \right\rangle \right| \\
=&\left| \left\langle T^{-1} \left(I- T \rho(h_i^{-1}) T^{-1}\right) P T \left(I- \rho(h_i)\right) \xi , \xi^* \right\rangle \right| ( \mathrm{Recall\ Ad}(T)\circ\rho \mid_{N} \mathrm{\ is\  unitary}), \\
=&\left| \left\langle \left(I- \rho(h_i^{-1})\right)T^{-1} P  T\left(I- \rho(h_i)\right) \xi , \xi^* \right\rangle \right| =\left| \left\langle T^{-1} P  T\left(\xi- \rho (h_i)\xi \right) , \xi^* - \rho^{\dagger}(h_i)\xi^* \right\rangle  \right| \\
\leq&  \left\| T^{-1} P  T\left(\xi- \rho(h_i)\xi\right)\right\|_{\rho} \left\| \xi^* - \rho^{\dagger}(h_i)(\xi^*) \right\|_{\rho *} \\
\leq & M \left\| T^{-1} P  T \right\|_{\mathbb{B}(\mathcal{H})} \left\|\xi- \rho (h_i)\xi\right\|_{\rho} \left\| \xi^* - \rho^{\dagger}(h_i)\xi^* \right\|_{\rho *} 
\leq   4M^{7}\varepsilon^2. 
\end{align*}
(In above, we remark that for $V \in \mathbb{B} (\mathcal{H})$, $V^*$ means the adjoint operator of $V$.)

Thanks to these two estimations, one can verify  
$\mu_+ (D_i) < 5M^6\varepsilon + 4M^7\varepsilon^2 $ for $1\leq i \leq4$. 
(Use for instance, $ \widehat{g_{1,-}}  (D_1 \cup D_2)$$\subset D_2 \cup D_0$, 
where $g_{1,-} =E_{1,2}(-1)$.)   
Hence the inequality  
$\mu_+ (\mathbb{T}^2 \setminus \{ 0 \}) \leq 20 M^6 \varepsilon + 20 M^7 \varepsilon^2 $ holds. 
If $\varepsilon \leq (21M^6)^{-1}$, then there must exist a non-zero $\rho (N)$-invariant vector. 
It is a contradiction.

For the general case, let us recall Kassabov's argument in \cite{Kas}. We 
identify $\widehat{A_k}$ with the set of all formal power series of variables 
$x_l^{-1}$ $(1 \leq l \leq k)$ over $\widehat{\mathbb{Z}} \cong \mathbb{T}$. 
Here the pairing is defined by
$$
\langle a x_1^{i_1} \cdots x_k^{i_k}  | \phi x_1^{-j_1} \cdots x_k^{-j_k}  \rangle  = \phi (a) \delta_{i_1 , j_1} \cdots \delta_{i_k , j_k}.
$$
We define the $valuation$ $v$ on $\widehat{A_k}$ as the minimum of the total 
 degrees of all terms. 
Here we naturally define $v(0)= +\infty$. 
We decompose $\hat{N} \setminus \{ 0 \} $$=\widehat{{A_k}^2} \setminus \{ 0 \}$ as follows: 
\begin{align*}
&A=\{ (\chi_1,\chi_2) : v(\chi_1) > v(\chi_2)>0 \}, 
\ B=\{ (\chi_1,\chi_2) : v(\chi_1) = v(\chi_2)>0 \}, \\
&C=\{ (\chi_1,\chi_2) : v(\chi_2) > v(\chi_1)>0 \}, 
\ D=\{ (\chi_1,\chi_2) : v(\chi_1)v(\chi_2)=0 \}. 
\end{align*}

Then from an argument similar to one in \cite{Kas}, 
we have the following inequalities: 
\begin{align*}
&\mu_+ (A)\leq  \mu_+ (D) + 5(k+1)M^6 \varepsilon , \\
&\mu_+(B) \leq \mu_+ (D) + 5kM^6 \varepsilon , \\
\mathrm{and}\ \  &\mu_+(C) \leq \mu_+ (D) + 5(k+1)M^6 \varepsilon.
\end{align*}
We naturally define the restriction map $\mathrm{res} \colon \hat{N} \to \hat{\mathbb{Z}}^2$ and obtain that $\mu_+ (D)= \mu_+ (\hat{N} \setminus \mathrm{res}^{-1} \{ 0 \} ) \leq  20 M^6 \varepsilon + 20 M^7 \varepsilon^2$. 
Finally, by combining these inequalities we conclude that
$$
1\leq \mu_+ (\hat{N})=\mu_+ (\hat{N} \setminus \{0 \})  \leq (15k+90)M^6 \varepsilon + 80 M^7 \varepsilon^2 . $$
(The middle equality in above follows from the assumption that $\mathcal{H}^{\rho (N)}=0$.) 
Hence in particular $\varepsilon $ must be more than $(15k+100)^{-1}M^{-6}$. 
\end{proof}

\bigskip

\begin{scriptsize}
\begin{center}
Graduate School of Mathematical Sciences, 
University of Tokyo, Komaba, Tokyo, 153-8914, Japan 

\noindent
e-mail: mimurac@ms.u-tokyo.ac.jp

\end{center}
\end{scriptsize}
\end{document}